\documentclass[12pt,twoside,a4paper,reqno]{amsart}

\usepackage{graphicx}
\usepackage{graphics}
\usepackage{amssymb}
\usepackage{eucal}
\usepackage{amsmath}
\usepackage{amsthm}
\newtheorem{theorem}{Theorem}[section]
\newtheorem{lemma}[theorem]{Lemma}

\newtheorem{corollary}[theorem]{Corollary}

\theoremstyle{definition}
\newtheorem{definition}[theorem]{Definition}
\newtheorem{example}[theorem]{Example}

\theoremstyle{remark}
\newtheorem{remark}[theorem]{Remark}

\numberwithin{equation}{section}

\begin{document}

\title{Continuous K-G-frames in Hilbert spaces}


\author{E. Alizadeh }
\address{Department of Mathematics, Shabestar Branch, Islamic Azad University, Shabestar, Iran 
}\email{esmaeil.alizadeh1020@gmail.com}  
\author{A. Rahimi}
\address{Department of Mathematics, University of Maragheh, Maragheh, Iran
} \email{rahimi@maragheh.ac.ir}
\author{ E. Osgooei }
\address{Faculty of Science, Urmia University of Technology, Urmia, Iran 
}\email{e.osgooei@uut.ac.ir }
\author{M. Rahmani}
\address{Young Researchers and Elite Club, Ilkhchi Branch, Islamic Azad University, Ilkhchi, Iran
}
 \email{m\_rahmani26@yahoo.com }





\begin{abstract}
In this paper, we  intend to introduce the  concept of c-K-g-frames, which are the  generalization of  K-g-frames. In addition, we  prove some new results on c-K-g-frames in Hilbert spaces. Moreover, we define the related operators of c-K-g frames. Then, we give necessary and sufficient conditions on c-K-g-frames to characterize them. Finally, we verify perturbation of c-K-g-frames.
\end{abstract}
\maketitle
\section{Introduction}
\label{intro}
Frames (discrete frames) were introduced  by Duffin and Schaeffer in 1952 \cite{7} for studying some profound problems in nonharmonic Fourier series. Discrete and continuous frames arise in many applications in both pure and applied mathematics and in particularly frame theory has been extensively used in many fields such as filter bank theory, signal and image processing, coding and communications \cite{15} and other areas.
 \par
 Over the years, various extensions of the frame theory have been investigated.
Several of these are contained as special cases of the elegant theory
for g-frames that was introduced in \cite{1001,16}. For example, one can consider: bounded quasi-projectors, fusion frames, pseudo-frames, oblique frames, outer frames and etc.
\par
Frames and their relatives are most often considered in the discrete case, for instance in signal processing \cite{7}. However, continuous frames have also been studied and offer interesting mathematical problems. They have been introduced originally by Ali, Gazeau and Antoine \cite{2} and also, independently, by Kaiser \cite{Kaiser}. Since then, several papers dealt with various aspects of the concept, see for instance \cite{Fornasier,9} or \cite{multi.con.frame,12,13,14}. By combining the above mentioned  extensions of frames, the new and more general notion called \textit{continuous g-frame} has been introduced in \cite{1}.
\par
Traditionally, frames were studied for the whole space or for the closed subspace.
Gavruta in \cite{k-frame} gave another generalization of frames namely K-frames, which allows to reconstruct elements from the range of a linear and
bounded operator in a Hilbert space. In general, range is not a closed subspace. K-frames allow us in a stable way, to reconstruct elements from the range of a linear and bounded operator in a Hilbert space.

 K-g-frames have been introduced in \cite{3,10} and some properties and characterizations of K-g-frames has been obtained, for more information on K-g-frames, the reader can check \cite{10,17}. In this paper, we generalize some results in \cite{10} and \cite{17} to continuous version of frames.
\par
Throughout this paper, $H, (\Omega, \mu)$ and $\{H_{\omega}\}_{\omega\in\Omega}$ will be a separable Hilbert space, a measure space and a family of Hilbert spaces, respectively and $K$ is a bounded linear operator on $H$.
At first we review some definitions and results about g-frames.
\begin{definition}
Let $K\in B(H)$. A sequence $\{f_{n}\}_{n=1}^{\infty}$ is called a K-frame for $H$, if there exist constants $A, B>0$ such that
\begin{equation}\label{b}
A\|K^{*}f\|^{2}\leq \displaystyle\sum_{n=1}^{\infty}|\langle f, f_{n}\rangle|^{2}\leq B\|f\|^{2},\quad f\in H.
\end{equation}
We call $A, B$ the lower  and the upper frame bounds of K-frame $\{f_{n}\}_{n=1}^{\infty}$, respectively.
If only the right inequality  (\ref{b}) holds, $\{f_{n}\}_{n=1}^{\infty}$ is called a Bessel sequence. If $K=I$, then it is just the ordinary frame.
\end{definition}
\vspace{2mm}
\begin{definition}
Let $K\in B(H)$ and $\Lambda=\{\Lambda_{i}\in B(H, H_{i}) : i\in I\}$. We call $\Lambda$ a K-g-frame for $H$ with respect to $\{H_{i}\}_{i\in I}$, or simply, a K-g-frame for $H$, if there exist constants $A, B>0$ such that
\begin{equation}
A\|K^{*}f\|^{2}\leq \displaystyle\sum_{i\in I}\|\Lambda_{i}f\|^{2}\leq B\|f\|^{2},\quad f\in H.
\end{equation}
The constants $A, B$ are called the lower and upper bounds of K-g-frame, respectively.
\end{definition}
\begin{remark}
 Every K-g-frame is a g-Bessel sequence for $H$. If $K=I$, K-g-frame is a g-frame.\\
\end{remark}

Denote the representation space $l^{2}\Big(\{H_{i}\}_{i\in I}\Big)$  by
$$l^{2}\Big(\{H_{i}\}_{i\in I}\Big)=\left\lbrace \{a_{i}\}_ {i\in I} : a_{i}\in H_{i},i\in I,\sum_{i\in I}\|a_{i} \|^{2}<\infty\right\rbrace .$$
We define the bounded linear operator $T: l^{2}\Big(\{H_{i}\}_{i\in I}\Big)\longrightarrow H$ as follows
\[T\Big(\{g_{i}\}_{i\in I}\Big)=\displaystyle\sum_{i\in I}\Lambda_{i}^{*}g_{i},\quad \{g_{i}\}_{i\in I}\in l^{2}\Big(\{H_{i}\}_{i\in I}\Big).\]
\begin{theorem}\em{(\cite{11}).}
The family $\Lambda=\{\Lambda_{i}\in B(H, H_{i}) :  i\in I\}$ is a g-Bessel sequence for $H$ with bound $B$ if and only if the operator
\begin{eqnarray*}
&~&T: l^{2}\Big(\{H_{i}\}_{i\in I}\Big)\longrightarrow H\\
&~&T\Big(\{g_{i}\}_{i\in I}\Big)=\displaystyle\sum_{i\in I}\Lambda_{i}^{*}g_{i}
\end{eqnarray*}
is a well-defined and bounded operator with $\|T\|\leq\sqrt{B}$.
\end{theorem}
\vspace{2mm}
The adjoint operator $T^{*}: H\longrightarrow l^{2}\Big(\{H_{i}\}_{i\in I}\Big)$ is given by \[T^{*}f=\{\Lambda_{i}f\}_{i\in I},\quad f\in H.\]
By composing $T$ with its adjoint $T^{*}$, we obtain the bounded linear operator $S: H\longrightarrow H$, as follows
\[Sf= TT^{*}f=\displaystyle\sum_{i\in I}\Lambda_{i}^{*}\Lambda_{i}f,\quad f\in H.\]
We call $T$, $T^{*}$ and $S$ the synthesis operator, analysis operator and frame operator of K-g-frame, respectively. These operators play important roles in studying of K-g-frame theory.

\begin{lemma}\label{k}\em{(\cite{7}).}
Let $H'$ be a complex Hilbert space. Suppose that $U: H\longrightarrow H'$ is a bounded linear operator with closed range $R(U)$. Then there exists a unique bounded linear operator $U^{\dagger}:H'\longrightarrow H$ satisfying
\[N(U^{\dagger})=R(U)^{\bot},~~R(U^{\dagger})=N(U)^{\bot},~~UU^{\dagger}f=f,~~f\in R(U).\]
The operator $U^{\dagger}$ is called the pseudo-inverse operator of $U$.
\end{lemma}
\vspace{2mm}
\begin{lemma}\em{(\cite{10}).}
Let $\{\Lambda_{i}\}_{i\in I}$ be a g-Bessel sequence for $H$  with respect  to $\{H_{i}\}_{i\in I}$. Then $\{\Lambda_{i}\}_{i\in I}$ is a K-g-frame for $H$ with respect to $\{H_{i}\}_{i\in I}$ if and only if there exists a constant $A>0$ such that $S\geq AKK^{*}$, where $S$ is the frame operator for $\{\Lambda_{i}\}_{i\in I}$.
\end{lemma}
\begin{definition}\em{(\cite{1}).}
Let
 $\Pi_{\omega\in\Omega} H_{\omega}=\{f: \Omega\longrightarrow\cup_{\omega\in \Omega} H_{\omega}:\;f(\omega)\in H_{\omega}\}.$
  We say that  $F\in\displaystyle\Pi_{\omega\in \Omega} H_{\omega} $ is strongly measurable if $F$ as a mapping of $\Omega$ to $\oplus_{\omega\in \Omega}H_{\omega}$ is measurable.
\end{definition}

Continuous  g-frames  are defined as below:
 \begin{definition}\label{f} \em{(\cite{1}).}
We say that $\Lambda=\{\Lambda_{\omega}\in B(H, H_{\omega}):~\omega\in\Omega\}$ is a continuous generalized frame or simply a continuous g-frame for $H$  with respect to $\{H_{\omega}\}_{\omega\in \Omega}$, if
\end{definition}
\begin{enumerate}
\item[(i)]
for each $f\in H, \{\Lambda_{\omega}f\}_{\omega\in\Omega}$ is strongly measurable,
\item[(ii)]
there exist two constants $A,B$ such that
\end{enumerate}
\begin{equation}\label{d}
A\|f\|^{2}\leq \int_{\Omega} \|\Lambda_{\omega}f\|^{2} d\mu (\omega)\leq B\|f\|^{2}~,~ f\in H.
\end{equation}
$A$ and $B$ are called the lower and upper continuous g-frame bounds, respectively.

\vspace{2mm}
\begin{theorem}\em{(\cite{1}).}
Let $\{\Lambda_{\omega}\}_{\omega\in\Omega}$ be a continuous g-frame for $H$ with respect to $\{H_{\omega}\}_{\omega\in\Omega}$ with frame bounds $A$, $B$. Then, there exists a unique positive and invertible operator $S: H\longrightarrow H$ such that for each $f, g\in H$,
\[\langle Sf, g \rangle=\int_{\Omega}\langle f, \Lambda^{*}_{\omega}\Lambda_{\omega}g \rangle d\mu(\omega)\]
and $AI\leq S\leq BI$.
\end{theorem}
\vspace{2mm}
\begin{definition}
Let
\begin{eqnarray*}
  && \Big(\oplus_{\omega\in \Omega}H_{\omega}, \mu\Big)_{L^{2}}\\
  &=&\left\lbrace  F\in \prod_{\omega\in \Omega} H_{\omega}: \textmd{F is strongly measurable}, \int_{\Omega} \|F(\omega)\|^{2}d\mu(\omega)<\infty\right\rbrace,
\end{eqnarray*}
with inner product given by
\[ \langle F, G \rangle=\int_{\Omega} \langle F(\omega), G(\omega) \rangle d\mu(\omega).\]
It can be proved that $\Big(\oplus_{\omega\in\Omega}H_{\omega}, \mu\Big)_{L^{2}}$ is a Hilbert space (\cite{1}). We will denote the norm of $F\in \Big(\oplus_{\omega\in\Omega}H_{\omega}, \mu\Big)_{L^{2}}$ by $\|F\|_{2}$.
\end{definition}
\vspace{2mm}
\begin{theorem}\label{g}\em{(\cite{1}).}
Let $\{\Lambda_{\omega}\}_{\omega\in\Omega}$ be a continuous g-Bessel family for $H$ with respect to $\{H_{\omega}\}_{\omega\in \Omega}$  with bound $B$. Then the mapping  $T$ of \linebreak $\Big(\oplus_{\omega\in\omega}H_{\omega}, \mu\Big)_{L^{2}}$ to $H$ defined by
\[\langle TF, g \rangle=\int_{\Omega}\langle \Lambda_{\omega}^{*}F(\omega), g \rangle d\mu(\omega), \quad F\in \Big(\oplus_{\omega\in\Omega}H_{\omega}, \mu\Big)_{L^{2}},~g\in H,\]
is linear and bounded with $\|T\|\leq\sqrt{B}$. Furthermore for each $g\in H$ and $\omega\in \Omega$,
\[T^{*}(g)(\omega)=\Lambda_{\omega}g.\]
\end{theorem}
 \section{Continuous K-g-frame}

In this section, we introduce the continuous version of K-g-frames.

\begin{definition}
Suppose that $(\Omega, \mu)$ is a measure space with positive measure $\mu$ and $K\in B(H)$. A family $\Lambda=\{\Lambda_{\omega}\in B(H, H_{\omega}):~\omega\in\Omega\}$, which $\{H_{\omega}\}_{\omega\in\Omega}$ is a family of Hilbert spaces, is called a continuous K-g-frame, or simply, a c-K-g-frame for $H$ with respect to $\{H_{\omega}\}_{\omega\in \Omega}$, if
\begin{enumerate}
\item[(i)]
for each $f\in H$; $\{\Lambda_{\omega}f\}_{\omega\in\Omega}$ is strongly measurable,
\item[(ii)]
there exist constants  $0<A\leq B<\infty$ such that
\begin{equation}\label{e}
A\|K^{*}f\|^{2}\leq \int_{\Omega} \|\Lambda_{\omega}f\|^{2} d\mu (\omega)\leq B\|f\|^{2},~~ f\in H.
\end{equation}
The constants $A$, $B$ are called lower and upper c-K-g-frame bounds, respectively.
\end{enumerate}
If $A$, $B$ can be chosen such that $A=B$, then $\{\Lambda_{\omega}\}_{\omega\in\Omega}$ is called a tight c-K-g-frame and if $A=B=1$, it is called Parseval c-K-g-frame. A family $\{\Lambda_{\omega}\}_{\omega\in\Omega}$ is called a c-g-Bessel family if the right hand inequality in (\ref{e}) holds. In this case, $B$ is called the Bessel constant.
\end{definition}
\vspace{2mm}
 Every c-K-g-frame is a c-g-Bessel family for $H$ with respect to $\{H_{\omega}\}_{\omega\in\Omega}$.
When $K=I$, a c-K-g-frame is a c-g-frame as defined in Definition \ref{f}.
\begin{example}
Suppose that $H$ is an infinite dimensional separable Hilbert space and $\{e_n\}_{n=1}^{\infty}$ is an orthonormal basis for $H$. Define the operator $K\in B(H)$ as follow:
$$K e_{2n}=e_{2n}+e_{2n-1}; ~~K e_{2n-1}=0, ~~n=1,2,...~.$$

For each $f\in H$, we have
\begin{eqnarray*}
Kf=K\big(\sum_{n=1}^{\infty}\langle f,e_n\rangle e_n\big)
&=&K\big(\sum_{n=1}^{\infty}\langle f,e_{2n}\rangle e_{2n}+\sum_{n=1}^{\infty}\langle f,e_{2n-1}\rangle e_{2n-1}\big)
\\&=&\sum_{n=1}^{\infty}\langle f,e_{2n}\rangle (e_{2n}+e_{2n-1}).
\end{eqnarray*}
By an easy calculation, the adjoint operator $K^*$ is given by
$$K^*f=\sum_{n=1}^{\infty}\langle f,e_{2n}+e_{2n-1}\rangle  e_{2n},~~f\in H.$$
Also,
\begin{eqnarray*}
\|K^*f\|^2 &=&\big\|\sum_{n=1}^{\infty}\langle f,e_{2n}+e_{2n-1}\rangle  e_{2n} \big\|^2
=\sum_{n=1}^{\infty} | \langle f,e_{2n}+e_{2n-1}\rangle  |^2
\\&&\leq 2\sum_{n=1}^{\infty} | \langle f,e_{2n}\rangle|^2+2\sum_{n=1}^{\infty} | \langle f,e_{2n-1}\rangle|^2
\leq 4\|f\|^2,
\end{eqnarray*}
So
$$ \|K^*f\|^2 \leq \sum_{n=1}^{\infty} | \langle f,e_{2n}+e_{2n-1}\rangle |^2 \leq 4\|f\|^2,$$
that is, $\{f_n\}_{n=1}^{\infty}=\{e_{2n}+e_{2n-1}\}_{n=1}^{\infty}$ is a $K$-frame for $H$.
Now, let $(\Omega,\mu)$ be a $\sigma$-finite measure space with infinite measure and $\{H_{\omega}\}_{\omega\in\Omega}$ be a family of Hilbert spaces. Since $\Omega$ is $\sigma$-finite, it can be written as a disjoint union  $\Omega=\bigcup \Omega_k$ of countably many subsets $\Omega_k\subseteq \Omega$ such that $\mu(\Omega_k)<\infty$ for all $k\in \mathbb{N}$. Without less of generality, assume that $\mu(\Omega_k)>0$ for all $k\in \mathbb{N}$. For each $\omega \in \Omega$, define the operator $\Lambda_\omega : H\longrightarrow H_\omega$ by
$$\Lambda_\omega (f)=\frac{1}{\mu(\Omega_k)} \langle f,f_k\rangle h_\omega,~ f\in H, $$
where $k$ is such that $\omega \in \Omega_k$ and $h_\omega$ is an arbitrary element of $H_\omega$ such that $\|h_\omega\|=1$. For each $f\in H$, $\{\Lambda_\omega f\}_{ \omega\in \Omega}$ is strongly measurable (since $h_\omega$'s are fixed) and
$$\int_{\Omega} \|\Lambda_\omega f\|^2 d\mu (\omega) = \sum_{n=1}^{\infty} | \langle f,f_n \rangle  |^2.$$
Therefore
$$\|K^*f\|^2 \leq \int_{\Omega} \|\Lambda_\omega f\|^2 d\mu (\omega) = \sum_{n=1}^{\infty} | \langle f,f_n \rangle|^2 \leq 4\|f\|^2,$$
that is, $\{\Lambda_\omega f\}_{ \omega\in \Omega}$ is  a c-K-g-frame for $H$ with respect to $\{H_{\omega}\}_{\omega\in \Omega}$.
\end{example}

\begin{definition}
Suppose that $\{\Lambda_{\omega}\}_{\omega\in\Omega}$ is a c-K-g-frame for $H$ with respect to $\{H_{\omega}\}_{\omega\in\Omega}$ with frame bounds $A$, $B$. We define $S: H\longrightarrow H$ by
\begin{equation*}
\langle Sf, g \rangle= \int_{\Omega} \langle f, \Lambda_{\omega}^{*} \Lambda_{\omega}g \rangle d\mu (\omega),~ f, g\in H,
\end{equation*}
and we call it the c-K-g-frame operator.
\end{definition}
\vspace{2mm}
The following Lemma characterizes a c-K-g-frame by its frame operator.
\begin{lemma}\label{x}
Let $\{\Lambda_{\omega}\}_{\omega\in \Omega}$ be a c-g-Bessel family  for $H$ with respect to $\{H_{\omega}\}_{\omega\in\Omega}$. Then $\{\Lambda_{\omega}\}_{\omega\in\Omega}$ is a c-K-g-frame for $H$ with respect to $\{H_{\omega}\}_{\omega\in\Omega}$ if and only if there exists constant $A>0$ such that $S\geq AKK^{*}$, where $S$ is the frame operator of $\{\Lambda_{\omega}\}_{\omega\in \Omega}$.
\end{lemma}
\begin{proof}
The family $\{\Lambda_{\omega}\}_{\omega\in \Omega}$ is a c-K-g-frame for $H$ with bound $A$, $B$ if and only if
\begin{equation*}
A\|K^{*}f\|^{2}\leq \int_{\Omega} \|\Lambda_{\omega}f\|^{2} d\mu (\omega)\leq B\|f\|^{2},~~ f\in H,
\end{equation*}
that is\[\langle AKK^{*}f, f \rangle  \,  \leq ~ \langle Sf, f \rangle \, \leq \langle Bf, f \rangle,\quad f\in H,\]
where $S$ is the c-K-g-frame operator of $\{\Lambda_{\omega}\}_{\omega\in \Omega}$. Therefore, the conclusion holds.
\end{proof}
\vspace{2mm}
\begin{theorem}\label{h}
Let $(\Omega, \mu)$ be a measure space, where $\mu$ is $\sigma$-finite. Suppose that $\{\Lambda_{\omega}\in B(H, H_{\omega})~:~\omega\in\Omega\}$ is a family of operators such that  $\{\Lambda_{\omega}f\}_{\omega\in \Omega}$ is strongly measurable for each $f\in H$ and $K\in B(H)$. Then $\{\Lambda_{\omega}\}_{\omega\in \Omega}$ is a c-K-g-frame for $H$ with respect to $\{H_{\omega}\}_{\omega\in \Omega}$ if and only if the operator
\[T: \Big(\oplus_{\omega\in \Omega}H_{\omega}, \mu\Big)_{L^{2}}\longrightarrow H\]
weakly defined by
\[\langle TF, g \rangle= \int_{\Omega} \langle\Lambda^{*}_{\omega}F(\omega), g \rangle d\mu(\omega),~F\in \Big(\oplus_{\omega\in\Omega}H_{\omega}, \mu\Big)_{L^{2}},~~g\in H,\]
is bounded and $R(K)\subseteq R(T)$.
\end{theorem}
\begin{proof}
Let $\{\Lambda_{\omega}\}_{\omega\in\Omega}$ be a c-K-g-frame. By Theorem \ref{g}, $T$ is a bounded operator.
By Lemma \ref{x}, for each $f\in H$,
\[ \langle AKK^*f, f \rangle =A\|K^{*}f\|^{2}\leq \, \langle Sf, f \rangle=\langle TT^*f, f \rangle.\]
Thus  by Theorem 1 of \cite{6}, $R(K)\subseteq R(T)$.
 Now suppose $T$ is a bounded and $R(K)\subseteq R(T)$.
 Let $g\in H$, $F\in\Big(\oplus_{\omega\in\Omega}H_{\omega}, \mu\Big)_{L^{2}}$, then
\[\langle F, T^{*}g \rangle=\langle TF, g \rangle=\int_{\Omega}\langle\Lambda_{\omega}^{*}F(\omega), g\rangle d\mu(\omega).\]
By Lemma 2.10 of \cite{1}, $T^{*}g\in\Big(\oplus_{\omega\in\Omega}H_{\omega}, \mu\Big)_{L^{2}}$, where $T^{*}g(\omega)=\Lambda_{\omega}g$, $\omega\in \Omega$. Therefore
\[\|T^{*}g\|^{2}=\int_{\Omega}\|\Lambda_{\omega}g\|^{2}d\mu(\omega)\leq \|T\|^{2}\|g\|^{2},\]
so $\{\Lambda_{\omega}\}_{\omega\in\Omega}$ is a c-g-Bessel. Now we show that $\{\Lambda_{\omega}\}_{\omega\in\Omega}$ has the lower  c-K-g-frame  condition. Since $R(K)\subseteq R(T)$, by Theorem 1 of \cite{6}, there exists  $A>0$,  such that
\[KK^{*}\leq ATT^{*}.\]
Hence,
\[\dfrac{1}{A}\|K^{*}f\|^{2}\leq \|T^{*}f\|^{2}=\int_{\Omega}\|\Lambda_{\omega}f\|^{2}d\mu(\omega),\quad f\in H.\]
\end{proof}
\vspace{2mm}
\noindent Note that, the frame operator of a c-K-g-frame is not positive and invertible on $H$ in general, but we can show that under some conditions, it is positive and invertible on the subspace $R(K)\subseteq H$. In fact, by assumption, $R(K)$ is closed, due to Lemma \ref{k}, there exists  pseudo-inverse $K^{\dagger}$ of $K$ such that $KK^{\dagger}f=f$, $ f\in R(K)$, therefore, $S$ is positive and invertible.

\section{Main results of c-K-g-frames}
At first, we give the following equivalent characterization of c-K-g-frames.

\begin{theorem}\label{z}
Let $K\in B(H)$. Then the following statements are equivalent.
\begin{enumerate}
\item[(1)]
$\{\Lambda_{\omega}\}_{\omega\in\Omega}$ is a c-K-g-frame for $H$ with respect to $\{H_{\omega}\}_{\omega\in\Omega}.$
\item[(2)]
$\{\Lambda_{\omega}\}_{\omega\in\Omega}$ is a c-g-Bessel family for $H$ with respect to $\{H_{\omega}\}_{\omega\in\Omega}$ and there exists a c-g-Bessel  family $\{\Gamma_{\omega}\}_{\omega\in\Omega}$ for $H$ with respect to $\{H_{\omega}\}_{\omega\in\Omega}$ such that
\begin{equation}\label{m}
[\langle Kf, h \rangle =\int_{\Omega}\langle\Lambda_{\omega}^{*}\Gamma_{\omega}f, h \rangle d\mu(\omega)~,~f, h\in H.]
\end{equation}
\end{enumerate}
\end{theorem}
\begin{proof}
$(1)\Rightarrow (2)$ There are constants $A, B>0$ such that
\[A\|K^{*}f\|^{2}\leq\int_{\Omega}\|\Lambda_{\omega}f\|^{2}d\mu(\omega)\leq B\|f\|^{2},~f\in H,\]
Since
\[KK^{*}\leq \dfrac{1}{A}T_{\Lambda}T_{\Lambda}^{*}\]
  by Theorem 1 in \cite{6}, there exists a bounded operator $\Gamma\in B\Big(H, (\oplus_{\omega\in\Omega}H_{\omega}, \mu)_{L^{2}}\Big)$ such that $K=T_{\Lambda}\Gamma$. Now  for each $\omega\in \Omega$ and $ h\in H$, define
\begin{eqnarray*}
&~&\Gamma_{\omega}: H\longrightarrow H_{\omega}\\
&~&\Gamma_{\omega}(h)=(\Gamma h)(\omega),
\end{eqnarray*}
we have
\[\int_{\Omega} \|\Gamma_{\omega}h \|^{2}d\mu(\omega)=\int_{\Omega}\|(\Gamma h)(\omega)\|^{2}d\mu(\omega)=\|\Gamma h\|^{2}\leq \|\Gamma\|^{2}\|h\|^{2}.\]
So $\{\Gamma_{\omega}\}_{\omega\in\Omega}$ is a c-g-Bessel family.
\\
Also, we have
\[Kf=T_{\Lambda}\Gamma f=T_{\Lambda}\Big(\{(\Gamma f)(\omega)\}_{\omega\in\Omega}\Big),\quad f\in H,\]
hence
\begin{equation*}
\langle Kf, h \rangle=\int_{\Omega} \langle \Lambda_{\omega}^{*}\Gamma_{\omega}f, h \rangle d\mu(\omega),~f, h\in H.
\end{equation*}
$(2)\Rightarrow (1)$ It is enough to show that $\{\Lambda_{\omega}\}_{\omega\in \Omega}$ has the lower frame condition. Set $f=K^{*}h$ in (\ref{m}), so
\begin{eqnarray*}
\langle KK^{*}h, h \rangle&=&|\langle K^{*}h, K^{*}h \rangle|=|\int_{\Omega}\langle \Lambda_{\omega}^{*}\Gamma_{\omega}K^{*}h, h\rangle d\mu(\omega)|\\&=&|\int_{\Omega}\langle \Gamma_{\omega}K^{*}h, \Lambda_{\omega}h\rangle d\mu(\omega)|\\
&\leq &\int_{\Omega}|\langle \Gamma_{\omega}K^{*}h, \Lambda_{\omega}h\rangle| d\mu(\omega)\\
&\leq &\int_{\Omega}\|\Gamma_{\omega}K^{*}h\|\|\Lambda_{\omega}h\|d\mu(\omega)\\
&\leq&\Big(\int_{\Omega}\|\Gamma_{\omega}K^{*}h\|^{2}d\mu(\omega)\Big)^{\frac{1}{2}}\Big(\int_{\Omega}\|\Lambda_{\omega}h\|^{2}d\mu(\omega)\Big)^{\frac{1}{2}}\\
&\leq&(B\|K^{*}h\|^{2})^{\frac{1}{2}}\Big(\int_{\Omega}\|\Lambda_{\omega}h\|^{2}d\mu(\omega)\Big)^{\frac{1}{2}},
\end{eqnarray*}
where B is the Bessel constant of $\{\Gamma_{\omega}\}_{\omega\in\Omega}$.
Hence
\[\dfrac{1}{\sqrt{B}}\|K^{*}h\|\leq\Big(\int_{\Omega} \|\Lambda_{\omega}h\|^{2}d\mu(\omega)\Big)^{\frac{1}{2}}.\]
\end{proof}
Note that $\{\Lambda_{\omega}\}_{\omega\in\Omega}$ and $\{\Gamma_{\omega}\}_{\omega\in\Omega}$ are not interchangeable in general. However, if we strengthen the condition, there exists another type of dual such that $\{\Lambda_{\omega}\}_{\omega\in\Omega}$ and a family $\{\Theta_{\omega}\}_{\omega\in\Omega}$ introduced by $\{\Gamma_{\omega}\}_{\omega\in\Omega}$ are interchangeable in the subspace $R(K)$ of $H$.
\begin{theorem}
Let $K\in B(H)$, $K$ be with closed range, $\{\Lambda_{\omega}\}_{\omega\in\Omega}$ and $\{\Gamma_{\omega}\}_{\omega\in\Omega}$ be two c-g-Bessel families as in (\ref{m}). Then there exists a family $\{\Theta_{\omega}\}_{\omega\in\Omega}$ such that
\[\langle f, h \rangle=\int_{\Omega}\langle \Lambda_{\omega}^{*}\Theta_{\omega}f, h \rangle d\mu(\omega),~f\in R(K)~,~h\in H\]
where $\Theta_{\omega}=\Gamma_{\omega}\Big(K^{\dagger}|_{R(K)}\Big)$. Furthermore, $\{\Lambda_{\omega}\}_{\omega\in \Omega}$ and $\{\Theta_{\omega}\}_{\omega\in \Omega}$ are interchangeable for any $f\in R(K)$.
\end{theorem}
\begin{proof}
By the assumption, $R(K)$ is closed, there exists pseudo-inverse $K^{\dagger}$ of $K$ such that $f=KK^{\dagger}f~,~ f\in R(K)$. We have from (\ref{m}) that
\[\langle f, h \rangle=\langle KK^{\dagger}f, h \rangle=\int_{\Omega}\langle \Lambda_{\omega}^{*}\Gamma_{\omega}K^{\dagger}f, h \rangle d\mu(\omega),~f\in R(K),\]
where
 $\{\Lambda_{\omega}\}_{\omega\in\Omega}$
 and
  $\{\Gamma_{\omega}\}_{\omega\in\Omega}$
  are c-g-Bessel families for $H$ with respect to $\{H_{\omega}\}_{\omega\in\Omega}$ and satisfy (\ref{m}). Now, let $\Theta_{\omega}=\Gamma_{\omega}\Big(K^{\dagger}|_{R(K)}\Big)$. Since \linebreak
  $K^{\dagger}|_{R(K)}: R(K)\longrightarrow H$ and $\Gamma_{\omega}: H\longrightarrow H_{\omega}$, so we have $\Theta_{\omega}: R(K)\longrightarrow H_{\omega}$. For any $f\in R(K)$, $K^{\dagger}f\in H$, so
\[\int_{\Omega}\|\Theta_{\omega}f\|^{2}d\mu(\omega)=\int_{\Omega}\|\Gamma_{\omega}K^{\dagger}f\|^{2}d\mu(\omega)\leq B\|K^{\dagger}f\|^{2}\leq B\|K^{\dagger}\|^{2}\|f\|^{2}.\]
Hence, $\{\Theta_{\omega}\}_{\omega\in \Omega}$ is a c-g-Bessel family for $R(K)$ with respect to $\{H_{\omega}\}_{\omega\in \Omega}$. Now, we show that $\{\Lambda_{\omega}\}_{\omega\in \Omega}$  and $\{\Theta_{\omega}\}_{\omega\in \Omega}$ are interchangeable on $R(K)$. In fact, for any $f, h\in R(K)$, we have
\begin{eqnarray*}
\langle f, h \rangle= \int_{\Omega} \langle\Lambda_{\omega}^{*}\Theta_{\omega}f, h \rangle d\mu(\omega)
=\overline{\int_{\Omega}\langle \Theta_{\omega}^{*}\Lambda_{\omega}h, f \rangle d\mu(\omega)},
\end{eqnarray*}
that is
\[\langle h, f \rangle=\int_{\Omega}\langle \Theta_{\omega}^{*}\Lambda_{\omega}h, f\rangle d\mu(\omega).\]
Hence
\[\langle f, h \rangle=\int_{\Omega}\langle\Theta_{\omega}^{*}\Lambda_{\omega}f, h\rangle d\mu(\omega),\quad f, h\in R(K).\]
\end{proof}
\begin{theorem}
Let $K\in B(H)$ and $\{\Lambda_{\omega}\}_{\omega\in\Omega}$ be a c-g-frame for $H$ with respect to $\{H_{\omega}\}_{\omega\in\Omega}$. Then $\{\Lambda_{\omega}K^{*}\}_{\omega\in\Omega}$ is a c-K-g-frame for $H$ with respect to $\{H_{\omega}\}_{\omega\in\Omega}$.
\end{theorem}
\begin{proof}
By Theorem \ref{z}, it suffices to show that $\{\Lambda_{\omega}K^{*}\}_{\omega\in\Omega}$ is a c-g-Bessel family and there exists a c-g-Bessel family $\{\Gamma_{\omega}\}_{\omega\in\Omega}$ such that
\[\langle Kf, h \rangle=\int_{\Omega}\langle (\Lambda_{\omega}K^{*})^{*}\Gamma_{\omega}f, h \rangle d\mu(\omega),\quad f, h\in H,\]
By the assumption, $\{\Lambda_{\omega}\}_{\omega\in\Omega}$ is a c-g-frame, so by Theorem 2.4 in [1], we have
\begin{equation}\label{r}
\langle f, g \rangle=\int_{\Omega}\langle \Lambda_{\omega}^{*}\Lambda_{\omega}S^{-1}f, g\rangle d\mu(\omega),\quad f, g\in H,
\end{equation}
where $S$ is the c-g-frame operator for $\{\Lambda_{\omega}\}_{\omega\in\Omega}$ and $\{\Lambda_{\omega}S^{-1}\}_{\omega\in\Omega}$ is the canonical dual
of $\{\Lambda_{\omega}\}_{\omega\in\Omega}$. For any $f\in H$, we have $K^{*}f\in H$. Then, it follows that
\[\int_{\Omega}\|\Lambda_{\omega}K^{*}f\|^{2}d\mu(\omega)\leq B\|K^{*}f\|^{2}\leq B\|K^{*}\|^{2}\|f\|^{2}.\]
Therefore $\{\Lambda_{\omega}K^{*}\}_{\omega\in\Omega}$ is a c-g-Bessel family for $H$ with respect to $\{H_{\omega}\}_{\omega\in\Omega}$. By (\ref{r}), we have
\begin{eqnarray*}
\langle Kf, g\rangle &=&\langle f, K^{*}g\rangle=\int_{\Omega}\langle \Lambda_{\omega}^{*}\Lambda_{\omega}S^{-1}f, K^{*}g \rangle d\mu(\omega)\\
&=&\int_{\Omega}\langle K\Lambda_{\omega}^{*}\Lambda_{\omega}S^{-1}f, g \rangle d\mu(\omega)\\
&=&\int_{\Omega}\langle (\Lambda_{\omega}K^{*})^{*}\Lambda_{\omega}S^{-1}f, g \rangle d\mu(\omega),
\end{eqnarray*}
Set $\Gamma_{\omega}=\Lambda_{\omega}S^{-1}$, we see that
\[\langle Kf, g \rangle=\int_{\Omega}\langle (\Lambda_{\omega}K^{*})^{*}\Gamma_{\omega}f, g \rangle d\mu(\omega),~f, g\in H.\]
\end{proof}
\begin{theorem}
If $T, K\in B(H)$ and $\{\Lambda_{\omega}\}_{\omega\in\Omega}$ is a c-K-g-frame for $H$ with respect to $\{H_{\omega}\}_{\omega\in\Omega}$, then $\{\Lambda_{\omega}T^{*}\}_{\omega\in\Omega}$ is a c-$TK$-g-frame for $H$ with respect to $\{H_{\omega}\}_{\omega\in\Omega}$.
\end{theorem}
\begin{proof}
Since $\{\Lambda_{\omega}\}_{\omega\in\Omega}$ is a c-K-g-frame for $H$, by Definition \ref{e}, we have
\[A\|K^{*}f\|^{2}\leq\int_{\Omega}\|\Lambda_{\omega}f\|^{2}d\mu(\omega)\leq B\|f\|^{2},~f\in H.\]
Also,
\begin{eqnarray*}
A\|(TK)^{*}f \|^{2}&=&A\|K^{*}T^{*}f \|^{2}\leq\int_{\Omega}\|\Lambda_{\omega}T^{*}f\|^{2}d\mu(\omega)\\
&\leq&B \|T^{*}f\|^{2}\leq B \|T^{*}\|^{2}\|f\|^{2},\quad f\in H
\end{eqnarray*}
Therefore $\{\Lambda_{\omega}T^{*}\}_{\omega\in\Omega}$ is a c-$TK$-g-frame for $H$ with respect to $\{H_{\omega}\}_{\omega\in\Omega}$.
\end{proof}
\vspace{2mm}
\begin{corollary}
Let $K\in B(H)$. If $\{\Lambda_{\omega}\}_{\omega\in\Omega}$ is a c-K-g-frame for $H$ with respect to $\{H_{\omega}\}_{\omega\in\Omega}$, then $\{\Lambda_{\omega}(K^{*})^{N}\}_{\omega\in\Omega}$ is c-$K^{N+1}$-g-frame for $H$ with respect to $\{H_{\omega}\}_{\omega\in\Omega}$, where $N$ is a natural number.
\end{corollary}

\section{Perturbation of c-K-g-frames}
 Perturbation of discrete frames and frames associated with measurable spaces (c-frame) have been discussed  in \cite{4} and \cite{9}, respectively, and perturbations of g-frames, c-g-frames and their dual have been discussed in \cite{11}, \cite{18} and \cite{1}. Stability and perturbation of K-g-frames have been investigated in \cite{3}, \cite{19,20} and \cite{10}. In this section, we introduce perturbation of c-K-g-frames.

\begin{theorem}
Let  $\{\Lambda_{\omega}\}_{\omega\in \Omega}$ be a c-K-g-frame for $H$ with bound $A, B$ and $\{\Gamma_{\omega}\}_{\omega\in \Omega}$ be a family of operators such that  $\{\Gamma_{\omega}\}_{\omega\in \Omega}$  is strongly measurable for each $f\in H$. If there exist constants $\lambda_1,\lambda_2,\gamma\geq 0$ such that
$\max\{\lambda_2,\frac{\gamma}{A}+\lambda_1\}<1$ and for each $f, g \in H$,
\begin{align}\label{aa}
\int _{\Omega}\big|\langle (\Lambda_{\omega}^{*}\Lambda_{\omega}-\Gamma_{\omega}^{*}\Gamma_{\omega})f, g\rangle \big| d\mu(\omega)\hspace{3.5cm}
\end{align}
\begin{eqnarray*}
&\leq&\lambda_{1}\int_{\Omega} \big|\langle \Lambda_{\omega}^{*}\Lambda_{\omega}f, g\rangle \big| d\mu(\omega)+\lambda_{2}\int_{\Omega} \big|\langle \Gamma_{\omega}^{*}\Gamma_{\omega}f, g\rangle \big| d\mu(\omega) +\gamma \|K^*f\|^2,
\end{eqnarray*}
\\
\noindent
then $\{\Gamma_{\omega}\}_{\omega\in \Omega}$ is a c-K-g-frame for $H$ with respect to $\{H_{\omega}\}_{\omega\in \Omega}$ with bounds
\[\dfrac{(1-\lambda_{1})A-\gamma}{1+\lambda_{2}}~~\text{and}~~\dfrac{(1+\lambda_{1})B+\gamma\|K\|^2}{1-\lambda_{2}}.\]\\
\end{theorem}
\begin{proof}
For each $f, g\in H,$
\begin{eqnarray*}
\Big|\int_{\Omega}\langle \Gamma_{\omega}^{*}\Gamma_{\omega}f, g\rangle d\mu(\omega)\Big|&\leq
& \int_{\Omega}\big|\langle (\Lambda_{\omega}^{*}\Lambda_{\omega}-\Gamma_{\omega}^{*}\Gamma_{\omega})f, g\rangle \big| d\mu(\omega)\\
&+&\int_{\Omega}\big|\langle (\Lambda_{\omega}^{*}\Lambda_{\omega}f, g\rangle \big| d\mu(\omega)\\
&\leq&(\lambda_{1}+1)\int_{\Omega} \big|\langle \Lambda_{\omega}^{*}\Lambda
_{\omega}f, g\rangle \big| d\mu(\omega)\\
&+&\lambda_{2}\int_{\Omega} \big| \langle \Gamma_{\omega}^{*}\Gamma_{\omega}f, g\rangle \big| d\mu(\omega)+
\gamma \|K^*f\|^2.
\end{eqnarray*}
So for each $f, g\in H$,
\begin{eqnarray*}
\int_{\Omega} \big|\langle \Gamma_{\omega}^{*}\Gamma_{\omega}f, g\rangle \big|d\mu(\omega)
&\leq&\dfrac{1+\lambda_{1}}{1-\lambda_{2}}\int_{\Omega}\big|\langle \Lambda_{\omega}^{*}\Lambda_{\omega}f, g\rangle \big| d\mu(\omega)
+\dfrac{\gamma}{1-\lambda_{2}}\|K^*f\|^2.
\end{eqnarray*}
Hence for each  $f\in H$,
\begin{eqnarray*}
\int_{\Omega}\|\Gamma_{\omega}f\|^{2}d\mu(\omega)&\leq &\dfrac{1+\lambda_{1}}{1-\lambda_{2}}B\|f\|^{2}+\dfrac{\gamma}{1-\lambda_{2}}\|K\|^2 \|f\|^{2}\\
&\leq &\dfrac{(1+\lambda_{1})B+\gamma\|K\|^2}{1-\lambda_{2}}\|f\|^{2}.
\end{eqnarray*}
Therefore $\{\Gamma_{\omega}\}_{\omega\in \Omega}$ is a c-g-Bessel family for $H$.

Now we show that $\{\Gamma_{\omega}\}_{\omega\in\Omega}$ has the lower  c-K-g-frame  condition. For each $f\in H$,
\begin{align*}
\int_{\Omega}\|\Gamma_{\omega}f\|^{2}d\mu(\omega)
&=\int_{\Omega} \big|\langle \Gamma_{\omega}^{*}\Gamma_{\omega}f, f\rangle \big| d\mu(\omega)
\\&=\int_{\Omega} \big|\langle \Gamma_{\omega}^{*}\Gamma_{\omega}f-\Lambda_{\omega}^{*}\Lambda_{\omega}f+\Lambda_{\omega}^{*}\Lambda_{\omega}f, f\rangle \big| d\mu(\omega)
\\ &\geq \int_{\Omega} \big|\langle \Lambda_{\omega}^{*}\Lambda_{\omega}f, f\rangle \big| d\mu(\omega)
-\int_{\Omega} \big|\langle \Gamma_{\omega}^{*}\Gamma_{\omega}f-\Lambda_{\omega}^{*}\Lambda_{\omega}f, f\rangle \big| d\mu(\omega)
\\ & \geq  \int_{\Omega}\|\Lambda_{\omega}f\|^{2}d\mu(\omega)-\lambda_1  \int_{\Omega}\|\Lambda_{\omega}f\|^{2}d\mu(\omega)
\\& ~~~ -\lambda_2 \int_{\Omega}\|\Gamma_{\omega}f\|^{2}d\mu(\omega)-\gamma \|K^*f\|^2.
\end{align*}
Then for each $f\in H$,
\begin{align*}
\int_{\Omega}\|\Gamma_{\omega}f\|^{2}d\mu(\omega)
&\geq\frac{1}{1+\lambda_2}\Big[(1-\lambda_1)\int_{\Omega}\|\Lambda_{\omega}f\|^{2}d\mu(\omega)-\gamma \|K^*f\|^2\Big]
\\&\geq\frac{1}{1+\lambda_2}\Big[(1-\lambda_1)A\|K^*f\|^2-\gamma \|K^*f\|^2\Big]
\\&\geq\frac{(1-\lambda_1)A-\gamma}{1+\lambda_2}\|K^*f\|^2.
\end{align*}
\end{proof}




\end{document}